\documentclass{amsart}
\usepackage{amsmath,amsthm,amsfonts,amssymb,amscd,amsbsy}
\usepackage{caption,enumerate}
\usepackage{dsfont}
\usepackage{hyperref}

\hypersetup{
    pdftoolbar=true,
    pdfmenubar=true,
    pdffitwindow=false,
    pdfstartview={FitH},
    pdftitle={},
    pdfauthor={},
    pdfsubject={},
    pdfkeywords={}
    pdfnewwindow=true,
    colorlinks=true, 
    linkcolor=blue,
    citecolor=blue,
    urlcolor=black,
}

\newcommand{\Spec}{\operatorname{Spec}}

\newcommand{\dd}{\mathrm{d}}
\newcommand{\id}{\operatorname{Id}}
\newcommand{\Vol}{\operatorname{Vol}}
\newcommand{\vol}{\operatorname{vol}}

\newcommand{\Ric}{\operatorname{Ric}}
\newcommand{\scal}{\operatorname{scal}}

\renewcommand{\div}{\operatorname{div}}

\newcommand{\Ss}{\mathds S}
\newcommand{\Hr}{\mathds H}

\newcommand{\R}{\mathds R}
\newcommand{\C}{\mathds C}

\newcommand{\SO}{\mathsf{SO}}
\newcommand{\SU}{\mathsf{SU}}
\newcommand{\U}{\mathsf{U}}
\newcommand{\Sp}{\mathsf{Sp}}
\newcommand{\Spin}{\mathsf{Spin}}
\newcommand{\G}{\mathsf{G}}
\newcommand{\K}{\mathsf{K}}
\renewcommand{\H}{\mathsf H}

\newcommand{\Ad}{\operatorname{Ad}}
\newcommand{\g}{\mathrm g}

\newcommand{\h}{\mathrm h}

\newcommand{\gr}{\g_\mathrm{round}}

\allowdisplaybreaks

\newtheorem{theorem}{Theorem}[]

\newtheorem{proposition}[theorem]{Proposition}

\newtheorem{mainthm}{\sc Theorem}

\newtheorem{maincor}[mainthm]{\sc Corollary}

\theoremstyle{definition}
\newtheorem{definition}[theorem]{Definition}

\theoremstyle{remark}
\newtheorem{remark}[theorem]{Remark}

\title[Nonuniqueness of conformal metrics with constant $Q$-curvature]{Nonuniqueness of conformal metrics with constant $Q$-curvature}


\author[R. G. Bettiol]{Renato G. Bettiol}
\address{City University of New York (Lehman College) \newline
\indent Department of Mathematics  \newline
\indent 250 Bedford Park~Blvd W\newline
\indent Bronx, NY, 10468, USA }
\email{r.bettiol@lehman.cuny.edu}

\author[P. Piccione]{Paolo Piccione}
\address{
Universidade de S\~ao Paulo \newline
\indent Departamento de Matem\'atica \newline
\indent Rua do Mat\~ao, 1010 \newline
\indent S\~ao Paulo, SP, 05508-090, Brazil}
\email{piccione@ime.usp.br}

\author[Y. Sire]{Yannick Sire}
\address{
Johns Hopkins University\newline
\indent Krieger Hall \newline
\indent 3400 N.~Charles St. \newline
\indent Baltimore, MD, 21218, USA}
\email{sire@math.jhu.edu}

\allowdisplaybreaks
\numberwithin{equation}{section}
\numberwithin{theorem}{section}

\date{\today}

\begin{document}

\begin{abstract}
We establish several nonuniqueness results for the problem of finding complete conformal metrics with constant (fourth-order) $Q$-curvature on compact and noncompact manifolds of dimension $\geq5$. Infinitely many branches of metrics with constant $Q$-curvature, but without constant scalar curvature, are found to bifurcate from Berger metrics on spheres and complex projective spaces.
These provide examples of nonisometric metrics with the same constant negative $Q$-curvature in a conformal class with negative Yamabe invariant, echoing the absence of a Maximum Principle. We also discover infinitely many complete metrics with constant $Q$-curvature conformal to  $\Ss^m\times \R^d$, $m\geq4$, $d\geq1$, and $\Ss^m\times \Hr^d$, $2\leq d\leq m-3$; which give infinitely many solutions to the \emph{singular} constant $Q$-curvature problem on round spheres $\Ss^n$ blowing up along a round subsphere $\Ss^k$, for all $0\leq k<(n-4)/2$.
\end{abstract}

\maketitle


\section{Introduction}

The study of fourth-order conformal invariants of a Riemannian manifold $(M,\g)$ of dimension $n\geq3$ naturally leads to the definition of $Q$-curvature:
\begin{equation*}
Q_\g=\frac{1}{2(n-1)}\Delta_\g\scal_\g-\frac{2}{(n-2)^2}\big\|\Ric_\g\big\|^2+\frac{n^3-4n^2+16n-16}{8(n-1)^2(n-2)^2}\scal_\g^2,
\end{equation*}
where $\Delta_\g u=-\div_\g(\nabla u)$ is the (nonnegative) Laplacian operator. Following the seminal works of Branson~\cite{branson} and Paneitz~\cite{paneitz}, there has been great interest in understanding geometric and analytic properties of $Q_\g$. This remains a very active pursuit, as evidenced by developments even just over the last 3 years \cite{GurHanLin2016,gm,hang-yang-2015,HangYang2016,lin1,lin2,lin3}, see \cite{hang-yang-survey} for a survey.

Analogously to the Yamabe problem, a central question is whether $Q_\g$ can be made constant by using conformal deformations of $(M,\g)$, which corresponds to a fourth-order elliptic PDE \eqref{eq:constQcurvequation} on the conformal factor. Despite substantial progress regarding the existence of solutions, to our knowledge, the issue of uniqueness (or lack thereof) has only been inspected in a small number of geometric settings. Grunau, Ould Ahmedou, and Reichel~\cite{grunau} established the existence of a continuum of radially symmetric solutions on hyperbolic space $\Hr^n$, $n\geq5$, with ODE techniques. Using an involved perturbation argument and Mazzeo's microlocal analysis of elliptic edge operators, Li~\cite{gangli-thesis,gangli} obtained a continuum of solutions conformal to a perturbation of Poincar\'e-Einstein metrics. 
A blowing up sequence of solutions (with one bubble) 
was built by Wei and Zhao~\cite{wei-zhao} on spheres $\Ss^n$, $n\geq25$, with a certain non conformally flat metric, implying the existence of infinitely many solutions in that conformal class; see also \cite{hebey-robert,qing-raske}.

Nevertheless, in consonance with the plethora of results for the Yamabe problem, the scope of nonuniqueness in the constant $Q$-curvature problem ought to be much richer, both geometrically and topologically. The purpose of the present paper is to confirm this, by exhibiting extensive nonuniqueness phenomena on a wide class of compact and noncompact manifolds, using variational bifurcation theory and other topological methods. We restrict ourselves to manifolds of dimension $n\geq5$, as the low-dimensional cases $n=3$ and especially $n=4$ require a separate discussion.


We begin by establishing a criterion (Theorem~\ref{thm:abstrbifrisult}) to detect bifurcations along families $\g_t$ of metrics with constant $Q$-curvature and constant scalar curvature that admit \emph{horizontally Einstein} Riemannian submersions with minimal fibers. This criterion is well-suited to metrics $\g_t$ that form a \emph{canonical variation}, i.e., are obtained by rescaling the vertical directions of a submersion by a factor $0<t<+\infty$. We show that bifurcations for the constant $Q$-curvature problem are ubiquitous among such families near the degenerate limits $t=0$ and $t=+\infty$, see Propositions~\ref{prop:bifurcation0} and \ref{prop:bifurcationInfty}. A convenient framework to generate examples is provided by homogeneous fibrations, 
see Section~\ref{sec:homex}. 
Although our results yield many more examples (see e.g.~Proposition~\ref{prop:HomExamples}), to simplify the exposition, we now only state our findings on 
the so-called \emph{Hopf bundles}, where $\g_t$ are often referred to as \emph{Berger metrics}:

\begin{mainthm}\label{mainthm:A}
There exists an infinite sequence of bifurcating branches of metrics on $M$ with constant $Q$-curvature, but nonconstant scalar curvature, that issue from the Berger metrics $\g_t$ as $t\searrow0$ and/or $t\nearrow+\infty$, according to the table below.
\end{mainthm}
\begin{center}
\vspace{-0.2cm}
\begin{table}[ht]
\begin{tabular}{|c|c|c|}
\hline
\begin{tabular}{c}
Hopf bundle \\[2pt]
$F\longrightarrow M\longrightarrow B$
\end{tabular}\rule[-1.2ex]{0pt}{0pt} \rule{0pt}{2.5ex} & \begin{tabular}{c}
Infinitely many \\ bifurcations \\ as $t\searrow0$
\end{tabular} & \begin{tabular}{c}
Infinitely many\rule[-1.2ex]{0pt}{0pt} \rule{0pt}{2.5ex} \\ bifurcations \\ as $t\nearrow+\infty$
\end{tabular} \\
\hline \noalign{\smallskip} \hline 
$\Ss^1\to \Ss^{2q+1}\to \C P^q$\rule[-1.2ex]{0pt}{0pt} \rule{0pt}{2.5ex} &  no & if $q\geq 6$\\
$\Ss^3\to\Ss^{4q+3}\to \Hr P^q$\rule[-1.2ex]{0pt}{0pt} \rule{0pt}{2.5ex} &  if $q\geq1$ & if $q\geq 2$\\
$\C P^1\to \C P^{2q+1}\to \Hr P^q$\rule[-1.2ex]{0pt}{0pt} \rule{0pt}{2.5ex}   & if $q\geq2$ & if $q\geq 3$\\
$\Ss^7\to \Ss^{15}\to \Ss^8(1/2)$\rule[-1.2ex]{0pt}{0pt} \rule{0pt}{2.5ex} & yes & yes\\
\hline
\end{tabular}
\end{table}
\end{center}
\vspace{-0.4cm}

Some remarkable facts about the constant $Q$-curvature problem on closed manifolds can be observed with the above result. First, one obtains global examples, e.g., on odd-dimensional spheres, of metrics that have constant $Q$-curvature but do not have constant scalar curvature.
Second, nonuniqueness of constant $Q$-curvature may take place on conformal classes with \emph{negative} Yamabe invariant; as is the case of conformal classes of Berger metrics $\g_t$ for sufficiently large $t$, since $\scal_{\g_t}\searrow-\infty$ as $t\nearrow+\infty$ if $\dim F\geq2$. Third, some of these Berger metrics $\g_t$, with large $t$, simultaneously have $\scal_{\g_t}<0$ and $Q_{\g_t}<0$, see Appendix~\ref{app:berger}; and nonisometric conformal metrics with the same constant (negative) $Q$-curvature can exist even in this regime. For comparison, recall that a metric with constant negative scalar curvature is \emph{unique} in its conformal class by the Maximum Principle. Clearly, this argument is not applicable to the constant $Q$-curvature problem due to its fourth-order nature.

Leaving the realm of closed manifolds, the constant $Q$-curvature problem can also be posed on manifolds with boundary or noncompact manifolds. On the latter, the natural boundary condition is \emph{completeness} of the metric. For instance, this is trivially satisfied by metrics that descend to a compact quotient $M/\Gamma$; we call these \emph{periodic} solutions with period $\Gamma$. Our second main result exploits the abundance of discrete cocompact groups $\Gamma$ on symmetric spaces to find infinitely many periodic solutions \emph{with different periods}, reflecting a truly noncompact phenomenon:

\begin{mainthm}\label{mainthm:B}
Let $(C,\g)$ be a closed manifold with constant scalar curvature and $(N,\h)$ be a simply-connected symmetric space of noncompact or Euclidean type, such that $(C\times N,\g\oplus\h)$ has dimension $\geq5$, $\scal\geq0$ and $Q\geq0$ but $Q\not\equiv0$. Then $(C\times N,\g\oplus\h)$ has infinitely many nonhomothetic periodic conformal metrics with constant positive $Q$-curvature.
\end{mainthm}

An immediate consequence of the above is the existence of infinitely many complete metrics with constant $Q$-curvature conformal to the standard product metrics on $\Ss^m\times \R^d$ for all $m\geq4$, $d\geq1$, and on $\Ss^m\times \Hr^d$ for all $2\leq d\leq m-3$. 
Using the stereographic projection, it is easy to see that $\Ss^{n-1}\times\R$ is conformally equivalent to both $\Ss^{n}\setminus\{\pm p\}$ and $\R^n\setminus\{0\}$, endowed with their (incomplete) constant curvature metrics. Thus, there are also infinitely many solutions to the constant $Q$-curvature problem on $\Ss^{n}\setminus\{\pm p\}=\Ss^n\setminus\Ss^0$ and $\R^n\setminus\{0\}$, for all $n\geq5$.

A higher codimension version of this argument shows that, for all $1\leq k<n$, $\Ss^{n-k-1}\times\Hr^{k+1}$ is conformally equivalent to $\Ss^n\setminus\Ss^k$ endowed with the (incomplete) round metric, see~\cite{bps-jdg}.
Therefore, another consequence of Theorem~\ref{mainthm:B} is that:

\begin{maincor}\label{maincor:C}
There are infinitely many complete metrics with constant positive $Q$-curvature on $\Ss^n\setminus \Ss^k$, $0\leq k<\frac{n-4}2$, conformal to the round~metric.
\end{maincor}

It would be interesting to determine whether $k<\frac{n-4}2$ is the \emph{maximal} range of dimensions in which nonuniqueness occurs. Notably, it follows from a result of Chang, Hang, and Yang~\cite[Thm~1.2]{CHY} that if periodic solutions to the constant $Q$-curvature problem exist on $\Ss^n\setminus \Lambda$ with $Q>0$ and $\scal>0$, then $\dim\Lambda<\frac{n-4}2$. Recall that, by the above conformal equivalence $\Ss^n\setminus\Ss^k \cong \Ss^{n-k-1}\times\Hr^{k+1}$, $1\leq k<n$, the pullback of the standard product metric gives a (trivial) solution with
\begin{equation*}
Q=\frac{n}{8}\left(n^2-4n(k+1)+4k(k+2)\right) \;\;\text{ and } \;\;\scal=(n-1)(n-2k-2),
\end{equation*}
and note that $Q>0$ if and only if $k<\frac{n-4}{2}$ or $k>\frac{n}{2}$, while $\scal>0$ if and only if $k<\frac{n-2}{2}$.
Furthermore, $k<\frac{n-4}2$ is precisely the range of dimensions $k$ of limit sets of Kleinian groups associated to locally conformally flat closed manifolds for which Qing and Raske~\cite{qing-raske} establish $C^\infty$-compactness of the space of metrics with constant positive $Q$-curvature and positive scalar curvature. 

\subsection{About the proofs}
Although the methods used to prove Theorems~\ref{mainthm:A} and~\ref{mainthm:B} are adapted from nonuniqueness results for the Yamabe problem, the technical hurdles to implement them are substantially more challenging in the (fourth-order) constant $Q$-curvature problem.

The Yamabe parallel to Theorem~\ref{mainthm:A} is the main result in~\cite{bp-calcvar}, whose proof uses classical variational bifurcation criteria~\cite{smoller-wasserman,kielhofer} that rely on computing the Morse index of a metric $\g$ as a critical point of the total scalar curvature functional, also called Hilbert-Einstein functional. This can be accomplished by analyzing the spectrum of the Laplacian~$\Delta_\g$. Meanwhile, computing the Morse index of a critical point of the total $Q$-curvature functional~\eqref{eq:deftotQcurvfunct} requires  analyzing the spectrum of the Paneitz operator
\begin{equation*}
P_\g\psi=\Delta_\g^2 \psi +\tfrac{4}{n-2}\div_\g(\Ric_\g(\nabla\psi,e_i)e_i)-\tfrac{n^2-4n+8}{2(n-1)(n-2)}\div_\g(\scal_\g\nabla\psi)+\tfrac{n-4}{2}Q_\g \psi.
\end{equation*}
Inspired by a simplification due to Otoba and Petean~\cite{OtobaPetean1} of the techniques in \cite{bp-calcvar}, we overcome this difficulty by finding an appropriate geometric framework that, while not imposing too many topological restrictions, reduces $P_\g$ on basic functions to a quadratic polynomial on~$\Delta_\g$. This facilitates the required spectral analysis, and leads to Theorem~\ref{thm:abstrbifrisult}, of which Theorem~\ref{mainthm:A} is a special instance.

Theorem~\ref{mainthm:B} and Corollary~\ref{maincor:C} are, in turn, the $Q$-curvature doppelg\"angers of \cite[Thm~1.1, Cor.~1.2]{frankenstein}, proved by playing the Aubin inequality against the volume growth along an infinite tower of finite-sheeted coverings. While the volume estimate side of the argument is the same, the existence result that produces new solutions at each step is a recent breakthrough of Hang and Yang~\cite{HangYang2016}, involving a \emph{reverse Aubin-type inequality} for a new conformal invariant, see Proposition~\ref{prop:hangyang}. 

\subsection{Organization of the paper}
In Section~\ref{sec:prelim}, we provide an overview of the constant $Q$-curvature problem, including its variational aspects and related conformal invariants. The main bifurcation criterion (Theorem~\ref{thm:abstrbifrisult}) is proved in Section~\ref{sec:bifurcation}, together with its consequences for degenerating canonical variations (Propositions~\ref{prop:bifurcation0} and \ref{prop:bifurcationInfty}). Section~\ref{sec:homex} contains the Lie theoretic apparatus used to produce examples of homogeneous fibrations to which the above bifurcation criteria apply, along with the proof of Theorem~\ref{mainthm:A}. The proof of Theorem~\ref{mainthm:B} is given in Section~\ref{sec:noncompact}. Finally, explicit formulae for the $Q$-curvature of Berger metrics are provided in Appendix~\ref{app:berger}.

\subsection{Acknowledgements}
We would like to thank the referees for their extraordinary attention to detail in the revision of our paper, which led to several improvements (mainly in Propositions \ref{prop:bifurcation0} and \ref{prop:bifurcationInfty}) and for alerting us to reference~\cite{CHY}.

\section{\texorpdfstring{Preliminaries on the constant $Q$-curvature problem}{Preliminaries on the constant Q-curvature problem}}
\label{sec:prelim}

In this section, we recall the variational formulation of the constant $Q$-curvature problem, including first and second variations, as well as some Yamabe-type invariants, an Aubin-type inequality and an existence result of Hang and Yang~\cite{HangYang2016}.

\subsection{\texorpdfstring{Paneitz operator and $Q$-curvature}{Paneitz operator and Q-curvature}}
Let $(M,\g)$ be a Riemannian manifold of dimension $n\geq5$. Recall that the $Q$-curvature of the metric $\g$ is defined as
\begin{equation*}
Q_\g=\tfrac{1}{2(n-1)}\Delta_\g\scal_\g-\tfrac{2}{(n-2)^2}\big\|\Ric_\g\big\|^2+\tfrac{n^3-4n^2+16n-16}{8(n-1)^2(n-2)^2}\scal_\g^2,
\end{equation*}
where $\Delta_\g u=-\div_\g(\nabla u)$ is the (nonnegative) Laplace operator on $(M,\g)$. The Paneitz operator $P_\g$ is defined using a local $\g$-orthonormal frame $(e_i)_{i=1}^n$, as
\begin{equation*}
P_\g\psi=\Delta_\g^2 \psi +\tfrac{4}{n-2}\div_\g(\Ric_\g(\nabla\psi,e_i)e_i)-\tfrac{n^2-4n+8}{2(n-1)(n-2)}\div_\g(\scal_\g\nabla\psi)+\tfrac{n-4}{2}Q_\g \psi.
\end{equation*}
Fix a background metric $\g_0$ in $M$ and denote by $[\g_0]$ its conformal class. 
Writing conformal metrics $\g\in[\g_0]$ as $\g=u^{\frac4{n-4}}\,\g_0$, where $u\colon M\to\R$, $u>0$, the Paneitz operator satisfies the covariance property that for any $\psi\colon M\to\R$,
\begin{equation*}
P_{\g}\psi=u^{-\frac{n+4}{n-4}}P_{\g_0}(u\,\psi).
\end{equation*}
Thus, $Q_{\g}=\frac2{n-4}P_{\g}(1)=\frac2{n-4}\,u^{-\frac{n+4}{n-4}}\, P_{\g_0}(u)$, and the following expression holds for the $Q$-curvature in terms of the Paneitz operator:
\begin{equation}\label{eq:totalQcurvpaneitz}
P_{\g_0} u=\tfrac{n-4}2\,Q_\g\, u^{\frac{n+4}{n-4}}.
\end{equation}
Therefore, the constant $Q$-curvature equation for the metric $\g=u^\frac{4}{n-4}\,\g_0$ reads:
\begin{equation}\label{eq:constQcurvequation}
\phantom{\qquad\lambda\in\R.}
P_{\g_0}u=\lambda\, u^\frac{n+4}{n-4},\qquad\lambda\in\R.
\end{equation}
In particular, it follows by elliptic regularity that constant $Q$-curvature metrics (in a smooth conformal class) are smooth.

\subsection{Variational setup}\label{sub:varsetup}
For the remainder of this section, suppose that $M$ is closed. Consider the \emph{(normalized) total $Q$-curvature functional}
\begin{equation}\label{eq:deftotQcurvfunct}
\mathcal Q\colon [\g_0]\to\R,\qquad \mathcal Q(\g)=\Vol(M,\g)^\frac{4-n}{n}\int_M Q_\g \vol_\g.
\end{equation}
Note that $\mathcal Q$ is invariant under homotheties, i.e., $\mathcal Q(\mu\,\g)=\mathcal Q(\g)$ for all $\mu>0$.
Using \eqref{eq:totalQcurvpaneitz}, it is easy to see that $\g=u^\frac{4}{n-4}\,\g_0$ satisfies:
\[\tfrac{n-4}2\mathcal Q(\g)=\left(\int_Mu^\frac{2n}{n-4}\,\vol_{\g_0}\right)^\frac{4-n}n\int_M u\,P_{\g_0}u\, \vol_{\g_0}=\frac{E_{\g_0}(u)}{\ \Vert u\Vert_{L^{\frac{2n}{n-4}}(M,\vol_{\g_0})}^2},\]
where $E_{\g_0}$ is the quadratic functional associated to the Paneitz operator $P_{\g_0}$,
\[E_{\g_0}(u)=\int_M u\,P_{\g_0}u\, \vol_{\g_0}.\]
The first variation of the functional $\mathcal Q$ can be computed as follows \cite{gm}:
\begin{equation}\label{eq:derivativeQ}
\dd\mathcal Q(\g)\colon T_\g[\g_0]\cong C^\infty(M)\to\R, \quad \dd\mathcal Q(\g)\phi=\tfrac{n-4}{2}\int_M (Q_\g-\overline Q_\g)\phi \vol_\g,
\end{equation}
where $\overline Q_\g$ is the mean value of $Q_\g$. Thus, $\g\in[\g_0]$ is a critical point of $\mathcal Q$ if and only if it has constant $Q$-curvature.

\begin{remark}\label{rem:constrvarprob}
By scale-invariance, constant $Q$-curvature metrics can also be characterized as critical points of the functional $E_{\g_0}$ in $[\g_0]$ subject to the constraint
\begin{equation*}\label{eq:constraint}
\Vert u\Vert_{L^{\frac{2n}{n-4}}(M,\vol_{\g_0})}=1.
\end{equation*} 
It is easy to recover the constant $Q$-curvature equation \eqref{eq:constQcurvequation} as the Euler--Lagrange equation of this constrained variational problem, and compute the value $Q_\g$ of the $Q$-curvature of $\g=u^{\frac4{n-4}}\,\g_0$ in terms of the Lagrange multiplier associated to the critical point $u$.
\end{remark}

If $\g$ has constant $Q$-curvature, the second variation $\dd^2\mathcal Q(\g)$ is represented by the $4$th order \emph{Jacobi operator} obtained by linearizing \eqref{eq:constQcurvequation}, namely:
\begin{equation}\label{eq:jacobi}
J_\g\psi = \tfrac12 P_\g\psi-\tfrac{n+4}{4}Q_\g \psi.
\end{equation}
This is a Fredholm operator $J_\g\colon C^{j+4,\alpha}(M)\to C^{j,\alpha}(M)$ for all $j\ge0$, which is symmetric with respect to the $L^2$-inner product.

\subsection{Yamabe-type invariants}
Recall that the \emph{Yamabe invariant} of $(M,\g_0)$ is:
\begin{equation}\label{eq:defYamabeinvariant}
Y(M,\g_0)=\inf_{u\in C^\infty(M)\setminus\{0\}}\frac{\int_M u\,L_{\g_0}u\, \vol_{\g_0}}{\ \ \Vert u\Vert^2_{L^{\frac{2n}{n-2}}(M,\vol_{\g_0})}},
\end{equation}
where $L_{\g_0}=4\,\tfrac{n-1}{n-2}\Delta_{\g_0}+\scal_{\g_0}$ is the conformal Laplacian.

Analogous conformal invariants have been defined for the Paneitz operator and $Q$-curvature.
First, in analogy with the Yamabe invariant of a conformal class:
\begin{equation}\label{eq:defY4}
Y_4(M,\g_0)=\inf_{u\in C^\infty(M)\setminus\{0\}}\frac{E_{\g_0}(u)}{\ \ \Vert u\Vert^2_{L^{\frac{2n}{n-4}}(M,\vol_{\g_0})}}.
\end{equation}
It is well-known that the infimum in \eqref{eq:defYamabeinvariant} is always attained at some positive function $u>0$, and the corresponding conformal metric $\g=u^{\frac4{n-2}}\g_0$ has constant scalar curvature. However, unlike the second order case, if a minimizer $u$ exists for the right hand side of \eqref{eq:defY4}, it need not be positive and hence there may be no conformal metric associated to it. Thus, it is natural to also define:
\begin{equation}\label{eq:defY4+}
Y_4^+(M,\g_0)=\inf_{\stackrel{u\in C^\infty(M)}{u>0}}\frac{E_\g(u)}{\Vert u\Vert^2_{L^{\frac{2n}{n-4}}(M,\vol_{\g_0})}}=\tfrac{n-4}2\inf_{\g\in[\g_0]}\mathcal Q(\g).
\end{equation}
Clearly, $Y_4^+(M,\g_0)\ge Y_4(M,\g_0)$, and these invariants coincide in some special cases. For instance, if $n=\dim M\ge6$ and there exists $\g\in[\g_0]$, with $\scal_\g>0$ and $Q_\g>0$,\footnote{%
If $\dim M\ge6$, the existence of $\g\in[\g_0]$ with with $\scal_\g>0$ and $Q_\g>0$ is proved in \cite{GurHanLin2016} to be equivalent to $Y(M,\g_0)>0$ and $P_{\g_0}>0$.} then $Y_4(M,\g_0)=Y_4^+(M,\g_0)$, and the (positive) infimum in \eqref{eq:defY4} is attained by a positive function $u$ such that $u^{\frac4{n-4}}\g_0$ has positive constant $Q$-curvature, and everywhere positive scalar curvature~\cite{GurHanLin2016}.

Lastly, suppose $Y(M,\g_0)>0$ and $Q_{\g_0}$ is \emph{almost positive}, that is, $Q_{\g_0}\ge0$ everywhere and $Q_{\g_0}>0$ at some point. Although it is not known whether this implies $Y_4(M,\g_0)>0$, in this situation $\ker P_{\g_0}=\{0\}$, and the Green's function $G_{P_{\g_0}}$ is everywhere positive on $M\times M$. The inverse of $P_{\g_0}$ is the integral operator
\[G_{P_{\g_0}}f(p)=\int_MG_{P_{\g_0}}(p,q)f(q)\,\vol_{\g_0}(q).\]
A new conformal invariant was introduced by Hang and Yang~\cite{HangYang2016} in this context:
\begin{equation}\label{eq:defTheta4}
\Theta_4(M,\g_0)=\sup_{f\in L^{\frac{2n}{n+4}}\setminus\{0\}}
\frac{\int_M f\,G_{P_{\g_0}}f\, \vol_{\g_0}}{\Vert f\Vert^2_{L^{\frac{2n}{n+4}}(M,\vol_{\g_0})}}.
\end{equation}
In some sense, $\Theta_4(M,\g_0)$ is comparable with the reciprocal $1/Y_4^+(M,\g_0)$.
The advantages of considering this quantity are that if a maximizer $f$ for \eqref{eq:defTheta4} exists, then $f$ is smooth, does not change sign, and the conformal metric $f^{\frac4{n+4}}\,\g_0$ has constant $Q$-curvature.
Moreover, $\Theta_4(M,\g_0)$ can be used to prove the following existence result and \emph{reverse} Aubin-type inequality~\cite[Thm.~1.4, Lem.~2.1]{HangYang2016}. 

\begin{proposition}\label{prop:hangyang}
If $(M,\g_0)$ is a closed Riemannian manifold of dimension $n\ge5$, with $Y(M,\g_0)>0$, and $Q_{\g_0}$ almost positive, then
\begin{equation}\label{eq:newTheta4}
\Theta_4(M,\g_0)=\tfrac2{n-4}\,\sup_{\g\in[\g_0]} \frac{\int_MQ_{\g}\,\vol_\g}{\Vert Q_\g\Vert^2_{L^{\frac{2n}{n+4}}(M,\vol_\g)}}.
\end{equation}
The supremum in \eqref{eq:defTheta4} is attained at some smooth function $f\in C^\infty(M)$, and the conformal metric $f^{\frac4{n+4}}\,\g_0$ has constant $Q$-curvature.
Moreover,
\begin{equation}\label{eq:AubinTheta4}
\Theta_4(M,\g_0)\geq \Theta_4(\Ss^n,\gr),
\end{equation}
with equality if and only if $(M,\g_0)$ is conformally equivalent to $(\Ss^n,\gr)$.
\end{proposition}

\section{Bifurcation criteria using Riemannian submersions}
\label{sec:bifurcation}

We now define a notion of \emph{bifurcation} for families of metrics with constant $Q$-curvature, and establish sufficient conditions for this phenomena to take place on the total space of Riemannian submersions, as its fibers are collapsed or dilated.

\begin{definition}\label{def:bifurcation}
Let $M$ be a closed manifold of dimension $n\geq5$, and let $\g_t$, $t\in[a,b]$, be a $1$-parameter family of Riemannian metrics on $M$ with constant $Q$-curvature.
We say that $t_*\in(a,b)$ is a \emph{bifurcation instant} for the family $\g_t$ if there exists a sequence $t_k$ in $[a,b]$ and a sequence of \emph{nonconstant} smooth positive functions $u_k\colon M\to\mathds R$ such that:
\begin{enumerate}[(i)]
\item $\lim\limits_{k\to\infty}t_k=t_*$,
\item the conformal metric $u_k^{\frac4{n-4}}\,\g_{t_k}$ has constant $Q$-curvature for all $k$,
\item $\lim\limits_{k\to\infty}u_k=1$ in the Sobolev space $W^{2,2}(M)$. 
\end{enumerate}
The collection of metrics $u_k^{\frac4{n-4}}\,\g_{t_k}$ is called a \emph{bifurcating branch} for the family $\g_t$.
\end{definition}

\subsection{Branch regularity}
Ellipticity of the constant $Q$-curvature problem can be used to improve the convergence of bifurcating solutions $u_k$ in Definition~\ref{def:bifurcation}:

\begin{proposition}
Assume that $[a,b]\ni t\mapsto\g_t$ is continuous in the $C^s$-topology, with $s\ge4$, and let $\widetilde\g_k=u_k^{\frac4{n-4}}\,\g_{t_k}$ be a bifurcating branch for the family $\g_t$. Then:
\begin{enumerate}[\rm (i)]
\item $\lim\limits_{k\to\infty} Q_{\widetilde\g_k}=Q_{\g_{t_*}}$,
\item $\lim\limits_{k\to\infty}u_k=1$  in the Whitney $C^\infty$-topology.
\end{enumerate}
In particular, $\widetilde\g_k$ tends to $\g_{t_*}$ in the $C^s$-topology, and its scalar curvature function tends to $\scal_{\g_{t_*}}$ in the $C^{s-2}$-topology.
\end{proposition}

\begin{proof}
Since the metrics $\g_{t_k}$ converge to $\g_{t_*}$ in the $C^4$-topology, one obtains: 
\[
\lim_{k \to \infty} Q_{\g_{t_k}}=Q_{\g_{t_*}},\quad\text{and}\quad\lim_{k \to \infty} \scal_{\g_{t_k}}=\scal_{\g_{t_*}}\ \text{in the $C^2$-topology.}
\]
The coefficients of the differential operator $P_{\g_{t_k}}$ tend to those of $P_{\g_{t_*}}$ in the $C^{s-3}$-topology; in particular, they converge uniformly. Moreover, a simple integration by parts argument shows that the quadratic form $u\mapsto\int_M u\,P_\g u\, \vol_\g$ is continuous in the $W^{2,2}$-topology. 
Since $u_k$ converges to $1$ in $W^{2,2}(M)$, using the Sobolev embedding $W^{2,2}(M)\hookrightarrow L^\frac{2n}{n-4}(M)$, we have that
$\Vol(M,\widetilde\g_k)=\int_M u_k^{\frac{2n}{n-4}}\vol_{\g_{t_k}}$
converges to $\Vol(M,\g_{t_*})$. Using these observations and \eqref{eq:totalQcurvpaneitz}, it follows that
\[
Q_{\widetilde\g_k}=\tfrac{2}{n-4}\Vol(M,\widetilde\g_k)^{-1}\int_M u_k \, P_{\g_{t_k}} u_k \, \vol_{\g_{t_k}},
\]
converges to $Q_{\g_{t_*}}$, verifying (i).

Using the costant $Q$-curvature equation, we have that 
\begin{multline}\label{eq:lapluk}
\Delta_{\g_{t_k}}^2 u_k =-\tfrac4{n-2}\div_{\g_{t_k}}\!\big(\!\Ric_{\g_{t_k}}(\nabla u_k,e_i)e_i\big)\\+\tfrac{n^2-4n+8}{2(n-1)(n-2)}\div_{\g_{t_k}}\!(\scal_{\g_{t_k}}\!\nabla u_k)-\tfrac{n-4}2 Q_{\g_{t_k}}u_k+\tfrac{n-4}2 Q_{\widetilde\g_k}\,u_k^{\frac{n+4}{n-4}};
\end{multline}
As observed above, all coefficients of the differential operators appearing in \eqref{eq:lapluk} converge uniformly. Moreover, using the Sobolev embedding $W^{2,2}(M)\hookrightarrow L^\frac{2n}{n-4}(M)$, the right hand side of \eqref{eq:lapluk} converges to $0\in L^p(M)$, where $p=\frac{2n}{n+4}>1.$
Standard elliptic estimates (see for instance \cite[Thm.~9.14, p.\ 240]{GilbargTrudinger01}) imply that $u_k$ converge in $W^{4,p}(M)$, hence (ii) follows from a standard bootstrap argument.
\end{proof}

\subsection{Main bifurcation criterion}
In order to search for bifurcations, we restrict to families of metrics that admit a very special type of submersion defined as follows.

\begin{definition}
A Riemannian submersion $\pi\colon(M,\g_M)\to (B,\g_B)$ is \emph{horizontally Einstein} if the Ricci tensor of $(M,\g_M)$ splits as $\Ric_{\g_M}=\Ric_{\mathcal H}\oplus\Ric_{\mathcal V}$ pointwise according to the splitting $T_p M=\mathcal H_p\oplus\mathcal V_p$ into horizontal and vertical subspaces, and $\Ric_{\mathcal H}=\kappa\,\pi^*(\g_B)$, where the constant $\kappa$ is called the \emph{horizontal Einstein constant}.
\end{definition}

Our key detection tool for bifurcation is the following:

\begin{theorem}\label{thm:abstrbifrisult}
Let $M$ be a closed manifold with $n=\dim M\ge5$, and 
\begin{equation*}
\pi_t\colon(M,\g_t)\to (B,\g_B), \quad t\in[t_*-\varepsilon,t_*+\varepsilon],
\end{equation*}
be a $1$-parameter family of horizontally Einstein Riemannian submersions with minimal fibers. 
Assume that $\g_t$ has constant scalar curvature and constant $Q$-curvature for all $t$, and denote by $\kappa_t$ the horizontal Einstein constant of $\pi_t$.
Given an eigenvalue $\lambda>0$ of the Laplacian $\Delta_{\g_B}$ of the base manifold $(B,\g_B)$, set
\begin{equation}\label{eq:alphabetadef}
\alpha_t=\frac{(n^2-4n+8)\scal_{\g_t}-\,8\,\kappa_t(n-1)}{4(n-1)(n-2)}\qquad \text{and}\qquad\beta_t=-2\,Q_{\g_t},
\end{equation}
and assume that
\begin{equation*}
\tfrac12\lambda^2+\alpha_{t_*}\,\lambda+\beta_{t_*}=0 \qquad \text{and}\qquad \alpha'_{t_*}\,\lambda+\beta'_{t_*}\ne0,
\end{equation*}
where $\alpha'_{t_*}=\frac{\dd}{\dd t}\alpha_t\big|_{t=t_*}$ and $\beta'_{t_*}=\frac{\dd}{\dd t}\beta_t\big|_{t=t_*}$.
Then $t_*$ is a bifurcation instant for the family $\g_t$ of constant $Q$-curvature metrics. If, in addition, $\lambda\ne\frac{\scal_{\g_{t_*}}}{n-1}$, then the constant $Q$-curvature metrics in the bifurcation branch that are sufficiently close to $\g_{t_*}$ do not have constant scalar curvature.
\end{theorem}

\begin{proof}
The strategy we employ is inspired by a bifurcation criterion for the Yamabe problem on Riemannian submersions due to Otoba and Petean~\cite{OtobaPetean1}, which simplifies earlier results by the first and second named authors~\cite{bp-calcvar,bp-pacific}.

Given a smooth function $v\colon B\to\mathds R$, denote by $\overline v\colon M\to\R$ its horizontal lift, i.e., $\overline v=v\circ\pi_t$. Functions of the form $\overline v$ are called \emph{basic functions}. 
Similarly, given a vector field $V\in\mathfrak X(B)$, denote by $\overline V\in\mathfrak X(M)$ its horizontal lift, i.e., the unique vector field on $M$ which is $\pi_t$-related to $V$. From the definition of Riemannian submersion, the gradient $\nabla \overline v$ with respect to $\g_t$ is the horizontal lift of the gradient $\nabla v$ with respect to $\g_B$. Since the fibers of $\pi_t$ are minimal, $\div_{\g_t}\overline V$ is the horizontal lift of $\div_{\g_B}V$. In particular, $\Delta_{\g_t}\overline v$ is the horizontal lift of $\Delta_B v$; and $\Delta_{\g_t}^2\overline v$ is the horizontal lift of $\Delta_{\g_B}^2v$. Using that $\g_t$ is horizontally Einstein, $\operatorname{div}\big(\sum_i\Ric_{\g_t}(\nabla_{\g_t}\overline v,e_i)e_i\big)$ is the horizontal lift of $-\kappa_t\Delta_{\g_B}v$.

Consider the family $P_t$ of fourth-order linear differential operators on $B$ given~by:
\begin{equation*}
P_tv=\Delta_{\g_B}^2v+2\,\alpha_t\,\Delta_{\g_B}v-\tfrac{n-4}{4}\,\beta_t\,v.
\end{equation*}
Using the observations above, it follows that the horizontal lift of $P_tv$ is $P_{\g_t}\overline v$. Thus, for a positive smooth function $v\colon B\to\mathds R$, the conformal metric $\overline v^\frac4{n-4}\,\g_t$ has constant $Q$-curvature equal to $C$ if and only if $v$ satisfies:
\begin{equation}\label{eq:constQcurvbase}
P_tv=\tfrac{n-4}2\, C\, v^\frac{n+4}{n-4}.
\end{equation}
Observe that \eqref{eq:constQcurvbase} is an elliptic equation with \emph{subcritical} exponent, since $\dim B<n$.
Clearly, \eqref{eq:constQcurvbase} is the Euler--Lagrange equation for critical points in $C^\infty(B)$ of the quadratic functional $\mathcal E_t(v)=\int_B v\,P_tv\, \vol_{\g_B}$ subject to the constraint:\footnote{Recall that the fibers of a Riemannian submersion with minimal fibers have constant volume, and denote by $\mathfrak v_t$ the $\g_t$-volume of the fibers of $\pi_t$. Thus, for any continuous map $v\colon B\to\R$,
\begin{equation}\label{eq:L2normhorizontal}
\Vert\overline v\Vert_{L^\frac{2n}{n-4}(M,\vol_{\g_t})}=\mathfrak v_t^\frac{n-4}{2n}\,\Vert v\Vert_{L^\frac{2n}{n-4}(B,\vol_{\g_B})}.
\end{equation}
In particular, $C^\infty(B)\ni v\mapsto\overline v\in C^\infty(M)$ maps $L^\frac{2n}{n-4}$-spheres to $L^\frac{2n}{n-4}$-spheres.}
\begin{equation}\label{eq:constraintbase}
\Vert v\Vert_{L^\frac{2n}{n-4}(B,\vol_{\g_B})}=\textrm{const.}
\end{equation}
We may choose the value of this constant equal to $\Vol(B,\g_B)^\frac{n-4}{2n}$ for all $t$, in such way that \eqref{eq:constraintbase} is satisfied by the constant function $1$.

Recalling the characterization of constant $Q$-curvature metrics in $[\g_0]$ as constrained critical points (Remark~\ref{rem:constrvarprob}), we conclude that \emph{horizontality is a natural constraint} for the constant $Q$-curvature problem.\footnote{This can also be proved directly since, under these assumptions, $\dd\mathcal Q(\g_t)f$ vanishes whenever $f\colon M\to\mathds R$ is a function with zero average on the fibers of the submersion.}
In other words, conformal metrics $\g=\overline v^\frac{4}{n-4}\,\g_0$ such that $Q_\g$ is constant, where $\overline v$ is a basic function, are precisely the critical points of the restriction of the total $Q$-curvature functional \eqref{eq:deftotQcurvfunct} to the subset of basic functions. In particular, given such a conformal factor $\overline v$, denoting by $v\colon B\to\R$ the corresponding function on the base so that $\overline v=v\circ\pi_t$, the second variation of the quadratic functional $\mathcal E_t$ at the critical point $v$ is given by the restriction of the Jacobi operator $J_\g$ defined in \eqref{eq:jacobi}. We thus define the $1$-parameter family of Jacobi operators:
\begin{equation}\label{eq:defJacobibase}
\phantom{,\qquad\phi\in C^\infty(B).}
J_t\phi=\tfrac12P_t\phi-\tfrac{n+4}4Q_{\g_t}\phi,\qquad\phi\in C^\infty(B).
\end{equation}
Consider the restriction of $J_t$ to the tangent space to the sphere \eqref{eq:constraintbase} at $v=1$, which consists of functions $\phi\colon B\to\R$ such that $\int_B\phi\,\vol_{\g_B}=0$. Note that this space is invariant under $J_t$, and hence the eigenvalues of the restriction are exactly the eigenvalues of $J_t$ with nonconstant eigenfunctions.
Since $J_t=\tfrac12\Delta^2_{\g_B}+\alpha_t\Delta_{\g_B}+\beta_t$ is a polynomial in $\Delta_{\g_B}$, its eigenvalues are:
\[\tfrac12\lambda^2+\alpha_t\,\lambda+\beta_t,\]
where $\lambda$ is an eigenvalue of $\Delta_{\g_B}$. As the spectrum of $\Delta_{\g_B}$ is discrete, so is the spectrum of $J_t$ for all $t$.
The assumptions in the statement imply that the constant function $1$ is a degenerate critical point of $\mathcal E_{t_*}$ subject to the constraint \eqref{eq:constraintbase}, that its Morse index jumps at $t=t_*$, and that $1$ is a nondegenerate critical point of $\mathcal E_{t_*\pm\varepsilon}$ for $\varepsilon>0$ sufficiently small. The conclusion that $t_*$ is a bifurcation instant then follows from standard variational bifurcation results, see e.g.~\cite{smoller-wasserman,kielhofer}. As to the last claim, $t_*$ is \emph{not} a bifurcation instant for the constant scalar curvature problem if $\lambda\ne\frac{\scal_{\g_{t_*}}}{n-1}$, see~\cite[Cor.~4]{LPZ12b}. Thus, metrics sufficiently close to $\g_{t_*}$ and nonhomothetic to some $\g_t$ do not have constant scalar curvature.
\end{proof}

\begin{remark}
It follows from the above proof that conformal factors associated with bifurcating solutions detected by Theorem~\ref{thm:abstrbifrisult} are basic, i.e., constant along the fibers of the submersion $\pi_t$. Thus, the conformal deformations of $(M,\g_t)$ producing other metrics with constant $Q$-curvature only involve horizontal directions.
\end{remark}

\subsection{Infinite bifurcations from asymptotic behavior}\label{subsec:asymptbif}
We now specialize to $1$-parameter families $\pi_t\colon (M,\g_t)\to (B,\g_B)$ of Riemannian submersions obtained rescaling
the vertical space of a fixed Riemannian submersion $\pi\colon (M,\g) \to (B,\g_B)$ using the parameter $0<t<+\infty$.
Such a family $\pi_t$ is called the \emph{canonical variation} of $\pi$ in Besse~\cite[\S9 G]{besse}.
We give sufficient conditions for the above criterion (Theorem~\ref{thm:abstrbifrisult}) to apply along each element $t_\lambda$ of sequences that either accumulate at $0$ or $+\infty$. Geometrically (in Gromov-Hausdorff sense), these correspond respectively to families $(M,\g_{t})$ of metrics with constant $Q$-curvature that bifurcate infinitely many times either as they collapse to the base $(B,\g_B)$ or as they degenerate to a sub-Riemannian limit. We henceforth systematically ignore other bifurcation instants $t_*$ that may exist at a bounded distance away from $0$ or $+\infty$.

Let $\pi\colon (M,\g)\to (B,\g_B)$ be a Riemannian submersion with totally geodesic fibers isometric to $(F,\g_F)$, and denote by $\mathcal H$ and $\mathcal V$ the corresponding horizontal and vertical distributions, so that $T_pM=\mathcal H_p\oplus \mathcal V_p$ is a $\g$-orthogonal direct sum for all $p\in M$. Assume that $(B,\g_B)$ and $(F,\g_F)$ are Einstein manifolds, that is, there exist $\Lambda_B,\Lambda_F\in\R$ such that 
\begin{equation}\label{eq:BFEinstein}
\Ric_B=\Lambda_B \,\g_B \quad  \text{and}\quad \Ric_F=\Lambda_F\,\g_F.
\end{equation}
Set $n=\dim M$ and $l=\dim F$, and note that $n\geq l$.
Following Besse~\cite[\S 9.33]{besse}, denote by $(X_i)_{i=1}^{n-l}$ and $(U_j)_{j=1}^l$  $\g$-orthonormal bases of $\mathcal H$ and $\mathcal V$ respectively, and~let 
\begin{equation}\label{eq:AAs}
\begin{aligned}
(A_X,A_Y) &:=\sum_{i=1}^{n-l} \g(A_X X_i,A_Y X_i)=\sum_{j=1}^l \g(A_X U_j, A_Y U_j)\\
(AU,AV) &:=\sum_{i=1}^{n-l} \g(A_{X_i} U,A_{X_i} V),
\end{aligned}
\end{equation}
where $A$ is the Gray-O'Neill tensor $A_{Z} W=(\nabla_{Z_\mathcal H} W_\mathcal V)_\mathcal H +(\nabla_{Z_\mathcal H} W_\mathcal H)_\mathcal V$, $Z,W\in T_pM$.
Suppose there exist $\zeta,\eta\in\R$ such that, for all $X,Y\in\mathcal H$ and $U,V\in\mathcal V$,
\begin{equation}\label{eq:Adiag}
(A_X,A_Y)= \zeta \,\g(X,Y) \quad\text{and}\quad (AU,AV)= \eta \,\g(U,V).
\end{equation}

Under the above hypotheses, the canonical variation $\g_t=t\, \g_\mathcal V \oplus \g_\mathcal H$, $t>0$, is a metric such that $\pi_t\colon (M,\g_t)\to(B,\g_B)$ is a Riemannian submersion with totally geodesic fibers isometric to $(F,t\,\g_F)$, see \cite[Prop.~9.68]{besse}. Moreover, the Ricci tensor $\Ric_t$ of $(M,\g_t)$ satisfies, see \cite[Prop.~9.70]{besse}:
\begin{equation}\label{eq:RicciDecomp}
\begin{aligned}
\Ric_t(X,Y)&=(\Lambda_B-2\zeta t)\,\g(X,Y),\\ 
\Ric_t\,(U,V)&=(\Lambda_F+\eta t^2)\,\g(U,V),\\ 
\Ric_t(X,V)&=0,\\
\end{aligned}
\end{equation}
for all $X,Y\in\mathcal H$ and $U,V\in\mathcal V$, where $\eta l=\zeta (n-l)$, see~\cite[\S 9.37]{besse}. In particular, $(M,\g_t)$ has constant scalar curvature
\begin{equation}\label{eq:scalcanvar}
\scal_t=\frac{l \Lambda_F}{t}+\Lambda_B (n-l)-\eta l t,
\end{equation}
as well as constant $Q$-curvature
\begin{equation}\label{eq:Qcanvar}
\begin{aligned}
Q_t&=-\frac{2(n-l)(\Lambda_B-2\zeta t)^2}{(n-2)^2} -\frac{2 l}{(n-2)^2} \left(\frac{\Lambda_F}{t}+\eta t\right)^2\\
&\quad+\frac{n^3-4n^2+16n-16}{8(n-1)^2(n-2)^2}\left(\frac{l \Lambda_F}{t}+(n-l)\Lambda_B-\eta l t\right)^2.
\end{aligned}
\end{equation}
Clearly, the Riemannian submersion $\pi_t$ is horizontally Einstein, with horizontal Einstein constant $\kappa_t=\Lambda_B-2\zeta t$.

Using the above notation, we now establish conditions for the existence of infinitely many bifurcation instants accumulating at $0$ and $\infty$.

\begin{proposition}\label{prop:bifurcation0}
Suppose that one of the following dimensional assumptions holds:
\begin{enumerate}[\rm (D1)]
\item $5\leq n\leq 8$ and $l\geq3$, or
\item $n\geq 9$ and $l\geq2$.
\end{enumerate}
If $\Lambda_F>0$, then there exists a sequence $\{t_\lambda\}$ accumulating at~$0$ of bifurcation instants for the family $\g_t$. Moreover, for $t_\lambda$ sufficiently small, the bifurcating metrics sufficiently close to $\g_{t_\lambda}$ do not have constant scalar curvature.
\end{proposition}

\begin{proof}
The functions $\alpha_t$ and $\beta_t$ defined in \eqref{eq:alphabetadef} can be computed explicitly using \eqref{eq:scalcanvar} and \eqref{eq:Qcanvar}. From these computations, asymptotically as $t\searrow 0$, we have
\begin{equation*}
\alpha_t^2-2\beta_t\sim \frac{\big((n^4 + 64n -64)l-128(n-1)^2\big)l}{16(n-1)^2(n-2)^2}\frac{\Lambda_F^2}{t^2}+O\left(\frac1t\right).
\end{equation*}
Direct inspection shows that, under any of the dimensional assumptions (D1) or (D2), the numerator in the above leading term satisfies
\begin{equation}\label{eq:uglypoly}
a(n,l)=\big((n^4 + 64n -64)l-128(n-1)^2\big)l>0
\end{equation}
and hence $\alpha_t^2-2\beta_t\nearrow+\infty$ as $t\searrow0$. Thus, for $t>0$ sufficiently small, the equation 
\begin{equation}\label{eq:bifurcation}
\tfrac12\lambda^2+\alpha_t\lambda+\beta_t=0
\end{equation}
has two real solutions $\lambda_t^\pm=-\alpha_t\pm\sqrt{\alpha_t^2-2\beta_t}$.
Moreover, asymptotically as $t\searrow0$,
\begin{equation*}
\lambda_t^+\sim\frac{-(n^2-4n+8)l+\sqrt{a(n,l)}}{4(n-1)(n-2)}\frac{\Lambda_F}{t}+ O\left(1\right).
\end{equation*}
Either (D1) or (D2) implies that the above leading coefficient is positive and, since $\Lambda_F>0$, we have that $\lambda_t^+\nearrow+\infty$ as $t\searrow0$.
Thus, for  arbitrarily large $\lambda\in\Spec(\Delta_B)$, there exists a sufficiently small $t_\lambda>0$ such that $\lambda$ is a solution of \eqref{eq:bifurcation} with $t=t_\lambda$. 
Since $\Spec(\Delta_B)$ is unbounded,  the corresponding sequence $\{t_\lambda\}$ accumulates at $0$.
A direct computation shows that, asymptotically as $t\searrow0$,
\begin{equation*}
\alpha_t'\lambda_t^+ +\beta_t'\sim \frac{a(n,l)-(n^2-4n+8)l\sqrt{a(n,l)}}{16(n-1)^2(n-2)^2}\frac{\Lambda_F^2}{t^3}+O\left(\frac{1}{t^2}\right).
\end{equation*}
The above leading coefficient is positive whenever one of the dimensional assumptions are satisfied, so $\alpha_t'\lambda_t^+ +\beta_t'\nearrow+\infty$ as $t\searrow0$. Thus, Theorem~\ref{thm:abstrbifrisult} implies that $t_\lambda$ is a bifurcation instant for $\g_t$, provided $\lambda\in\Spec(\Delta_B)$ is sufficiently large. 

Finally, in order to prove the last claim, observe that asymptotically as $t\searrow0$,
\begin{equation*}
\lambda_t^+-\frac{\scal_t}{n-1} \sim\frac{\sqrt{a(n,l)}-n^2 l}{4(n-1)(n-2)}\frac{\Lambda_F}{t} +O(1) 
\end{equation*}
The above leading coefficient is nonzero if and only if $n\neq \frac{l}{2}+1$, which is again satisfied because of the dimensional assumptions.
Thus, for $t_\lambda$ sufficiently small, $\lambda\neq\frac{\scal_{t_\lambda}}{n-1}$ and hence the last claim also follows from Theorem~\ref{thm:abstrbifrisult}.
\end{proof}

\begin{proposition}\label{prop:bifurcationInfty}
Suppose that one of the dimensional assumptions {\rm (D1)} or {\rm (D2)} in Proposition~\ref{prop:bifurcation0} holds, or else that
\begin{enumerate}[\rm (D3)]
\item $n\geq 21$ and $l=1$.
\end{enumerate}
 Moreover, assume
that $\zeta>0$, $\eta>0$, and 
\begin{equation}\label{eq:etazetaineq}
\frac{\eta}{\zeta}>\frac{8(n-1)\sqrt{n-l}}{\sqrt{(n^3-4n^2+16n-16)l^2-16(n-1)^2 l }}.
\end{equation}
Then there exists a sequence $\{t_\lambda\}$ that diverges to $+\infty$ of bifurcation instants for the family $\g_t$. Moreover, for $t_\lambda$ sufficiently large, the bifurcating metrics sufficiently close to $\g_{t_\lambda}$ do not have constant scalar curvature.
\end{proposition}

\begin{proof}
Computing $\alpha_t$ and $\beta_t$ explicitly, it follows that asymptotically as $t\nearrow+\infty$,
\begin{equation*}
\alpha_t^2-2\beta_t\sim \frac{a(n,l)\eta^2+b(n,l)\eta \zeta + c(n,l)\zeta^2}{16(n-1)^2(n-2)^2} t^2+O(t),
\end{equation*}
where $a(n,l)$ is defined in \eqref{eq:uglypoly} and
\begin{equation}\label{eq:uglypoly2}
b(n,l)=-32 l \left(n^3-5n^2+12n-8\right), \quad c(n,l)=-512 (n-1)^2 \big(n-l-\tfrac12\big).
\end{equation}
Routine computations show that any of the dimensional assumptions (D1), (D2), or (D3) implies:
\begin{equation}\label{eq:abcsigns}
a(n,l)>0, \quad b(n,l)<0, \quad c(n,l)<0,
\end{equation}
and hence $\delta(n,l)=b(n,l)^2-4a(n,l)c(n,l)>0$.
Consider the homogeneous polynomial
\begin{equation*}
q_1(\eta,\zeta)=a(n,l)\eta^2+b(n,l)\eta \zeta + c(n,l)\zeta^2.
\end{equation*}
Since $\delta(n,l)>0$, there is a factorization
\begin{equation*}
\frac{q_1(\eta,\zeta)}{\zeta^2}=a(n,l)\left(\frac{\eta}{\zeta}-\rho_-\right)\left(\frac{\eta}{\zeta}-\rho_+\right),
\end{equation*}
where $\rho_\pm=\frac{-b(n,l)\pm\sqrt{\delta(n,l)}}{2a(n,l)}$. Thus, by \eqref{eq:abcsigns}, we have that $q_1(\eta,\zeta)>0$ if and only if $\frac\eta\zeta<\rho_-$ or $\frac\eta\zeta>\rho_+$. 
The latter condition is satisfied due to \eqref{eq:etazetaineq}, since
\begin{equation*}
\frac{8(n-1)\sqrt{n-l}}{\sqrt{(n^3-4n^2+16n-16)l^2-16(n-1)^2 l }}>\rho_+
\end{equation*}
under any of the dimensional assumptions  (D1), (D2), or (D3). Therefore, we have $\alpha^2_t-2\beta_t\nearrow+\infty$ as $t\nearrow+\infty$. Routine computations imply that 
\begin{equation*}
\lambda^+_t=-\alpha_t+\sqrt{\alpha_t^2-2\beta_t}\sim \frac{(n^2-4n+8)l \eta-16(n-1)\zeta +\sqrt{q_1(\eta,\zeta)}}{4(n-1)(n-2)}t+O(1)
\end{equation*}
hence $\lambda_t^+\nearrow+\infty$ as $t\nearrow+\infty$ if and only if 
$(n^2-4n+8)l \eta-16(n-1)\zeta +\sqrt{q_1(\eta,\zeta)}>0$.
This holds by our assumptions, because $\frac{16(n-1)}{(n^2-4n+8)l}<\rho_+$.
Thus, for arbitrarily large $\lambda\in\Spec(\Delta_B)$, there exists a sufficiently large $t_\lambda>0$ such that $\lambda$ is a solution of \eqref{eq:bifurcation} with $t=t_\lambda$. 
Since $\Spec(\Delta_B)$ is unbounded,  the corresponding sequence $\{t_\lambda\}$ diverges to $+\infty$.
Moreover, similar computations show that \eqref{eq:etazetaineq} implies $\alpha'_t\lambda_t^+ +\beta'_t\searrow -\infty$ as $t\nearrow+\infty$. Thus, Theorem~\ref{thm:abstrbifrisult} implies that $t_\lambda$ is a bifurcation instant for $\g_t$, provided $\lambda\in\Spec(\Delta_B)$ is sufficiently large.

Finally, the last claim follows from the fact that $\scal_t\searrow-\infty$ as $t\nearrow +\infty$ since $\eta>0$, hence $\lambda_t^+-\frac{\scal_t}{n-1} \nearrow+\infty$. In particular, for $t_\lambda$ sufficiently large, $\lambda\neq\frac{\scal_t}{n-1}$ and hence the claim also follows from Theorem~\ref{thm:abstrbifrisult}.
\end{proof}

\begin{remark}
Assumption \eqref{eq:etazetaineq} is stated (for convenience) in terms of the ratio of $\eta$ and $\zeta$, which are quantities related to the Gray-O'Neill $A$-tensor of the submersion $F\to M\to B$, see \eqref{eq:Adiag}. Nevertheless, \eqref{eq:etazetaineq} is a \emph{purely dimensional} restriction, that is, it only depends on the dimensions of $F$ and $B$ and not on the $A$-tensor of the bundle itself, since $\frac{\eta}{\zeta}=\frac{n-l}{l}$, see~\cite[\S 9.37]{besse}. 
\end{remark}

\begin{remark}
Submersions with $1$-dimensional fibers are included in Proposition~\ref{prop:bifurcationInfty}, but excluded altogether from Proposition~\ref{prop:bifurcation0}, because $l=1$ implies $\Lambda_F=0$. Under these conditions, $\alpha^2_t-2\beta_t$ remains bounded as $t\searrow0$, so there can be at most \emph{finitely many} bifurcation instants for $t$ near $0$.
\end{remark}

\section{Homogeneous examples with two isotropy summands}
\label{sec:homex}

In this section, we use homogeneous spaces to provide a wealth of examples of Riemannian submersions to which the results (Propositions~\ref{prop:bifurcation0} and \ref{prop:bifurcationInfty}) from the previous section apply. Given a triple of compact Lie groups $\H\subset\K\subset\G$, define
\begin{equation}\label{eq:homfib}
\pi\colon \G/\H\to \G/\K, \quad \pi(g\H)=g\K.
\end{equation}
Denote by $\mathfrak h\subset\mathfrak k\subset\mathfrak g$ the corresponding Lie algebras. Endow $\mathfrak g$ with a bi-invariant inner product, and let $\mathfrak m$ and $\mathfrak p$ be complements  such that the direct sums 
\begin{equation*}
\mathfrak g=\mathfrak k\oplus\mathfrak m \quad \text{ and }\quad \mathfrak k=\mathfrak h\oplus\mathfrak p
\end{equation*}
are orthogonal. There is a bijective correspondence between $\Ad(\K)$-invariant inner products on $\mathfrak m$ and homogeneous ($\G$-invariant) Riemannian metrics on $\G/\K$; and analogously for $\Ad(\H)$-invariant inner products on $\mathfrak p$ and homogeneous ($\K$-invariant) Riemannian metrics on $\K/\H$. The orthogonal direct sum of two such inner products is an inner product on $\mathfrak m\oplus\mathfrak p$ that corresponds to a homogeneous ($\G$-invariant) Riemannian metric $\g$ on $\G/\H$ for which \eqref{eq:homfib} is a Riemannian submersion with totally geodesic fibers isometric to $\K/\H$, see~\cite[Thm.~9.80]{besse}. Clearly, $\mathfrak m$ and $\mathfrak p$ are respectively identified with the horizontal and vertical spaces of this submersion. Assume that $\mathfrak m$ and $\mathfrak p$ are inequivalent irreducible $\Ad(\H)$-representations. Then, by Schur's Lemma, the $\Ad(\H)$-invariant pairings \eqref{eq:AAs} are multiples of the metric, i.e., \eqref{eq:Adiag} holds, and the Ricci tensor $\Ric_{\G/\H}\colon\mathfrak m\oplus\mathfrak p\to \mathfrak m\oplus\mathfrak p$ is block diagonal. Thus, \eqref{eq:BFEinstein} holds as a consequence of \eqref{eq:Adiag} and \cite[Prop.~9.70]{besse}.

\begin{proposition}\label{prop:HomExamples}
Let $\H\subset\K\subset\G$ be compact Lie groups as above, and assume that $\mathfrak m$ and $\mathfrak p$ are inequivalent irreducible $\Ad(\H)$-representations, $\K$ is not Abelian, and 
that $n=\dim\G/\H$ and $l=\dim\K/\H$ satisfy either {\rm (D1)} or {\rm (D2)} in Proposition~\ref{prop:bifurcation0}. Then there exists an infinite
sequence of branches of constant $Q$-curvature inhomogeneous metrics bifurcating from the (canonical variation) family $\g_t$ of $\G$-invariant metrics  on $\G/\H$ as $t\searrow0$. On each bifurcating branch, infinitely many of these inhomogeneous metrics do not have constant scalar curvature.
\end{proposition}

\begin{proof}
As observed above, \eqref{eq:BFEinstein} and \eqref{eq:Adiag} hold, since $\mathfrak m$ and $\mathfrak p$ are inequivalent irreducible $\Ad(\H)$-representations. Furthermore, as the fiber $F=\K/\H$ is a compact homogeneous space, its Einstein constant satisfies $\Lambda_F\geq0$. Equality holds if and only if $\K/\H$ is a flat manifold, which would imply $\K$ is Abelian.
Thus, $\Lambda_F>0$ and hence the entire statement follows from Proposition~\ref{prop:bifurcation0}.
\end{proof}

Triples of compact Lie groups $\H\subset\K\subset\G$ such that $\mathfrak m$ and $\mathfrak p$ are inequivalent irreducible $\Ad(\H)$-representations, as in Proposition~\ref{prop:HomExamples}, were classified by Dickinson and Kerr~\cite{kerr-dickinson}, with corrections by He~\cite{chenxu}. This classification requires the further assumptions\footnote{These assumptions are not needed for our results.} that $\G$ is simple and simply-connected or $\G=\SO(n)$, and $\H$ is connected, and produces several dozen examples including many infinite families.

\subsection{Hopf bundles}\label{subsec:hopfbundles}
A well-known and geometrically interesting subclass of homogeneous spaces $\G/\H$ as in \eqref{eq:homfib} with two isotropy summands are the \emph{Hopf bundles}:
\begin{center}
\begin{table}[ht]
\caption{The Hopf bundles and corresponding Lie groups}\label{tab:hopfbundles}
\begin{tabular}{|c|c|c|c|c|}
\hline
& $F\longrightarrow M\longrightarrow B$\rule[-1.2ex]{0pt}{0pt} \rule{0pt}{2.5ex} & $\G$  & $\K$ & $\H$ \\
\hline \noalign{\smallskip} \hline 
(i) & $\Ss^1\to \Ss^{2q+1}\to \C P^q$\rule[-1.2ex]{0pt}{0pt} \rule{0pt}{2.5ex} & $\SU(q+1)$ & $\mathsf S(\U(q)\U(1))$ & $\SU(q)$\\
(ii) & $\Ss^3\to\Ss^{4q+3}\to \Hr P^q$\rule[-1.2ex]{0pt}{0pt} \rule{0pt}{2.5ex} & $\Sp(q+1)$ & $\Sp(q)\Sp(1)$ & $\Sp(q)$\\
(iii) & $\C P^1\to \C P^{2q+1}\to \Hr P^q$\rule[-1.2ex]{0pt}{0pt} \rule{0pt}{2.5ex} & $\Sp(q+1)$ & $\Sp(q)\Sp(1)$ & $\Sp(q)\U(1)$\\
(iv) & $\Ss^7\to \Ss^{15}\to \Ss^8(1/2)$\rule[-1.2ex]{0pt}{0pt} \rule{0pt}{2.5ex} & $\Spin(9)$ & $\Spin(8)$ & $\Spin(7)$\\
\hline
\end{tabular}
\end{table}
\end{center}

It is immediate to verify that Proposition~\ref{prop:HomExamples} applies to the above Hopf bundles, provided the dimensional restrictions are satisfied; i.e., (i) is excluded, (ii) for all $q\geq1$, (iii) for all $q\geq2$, and (iv). Furthermore, Proposition~\ref{prop:bifurcationInfty} also applies to the above examples, provided that the dimensional restrictions, including \eqref{eq:etazetaineq}, are satisfied. This is the case on (i) for all $q\geq10$, on (ii) for all $q\geq3$, on (iii) for all $q\geq4$, and on (iv). Indeed, \eqref{eq:etazetaineq} can be evaluated using the following table:
\begin{center}
\begin{table}[ht]
\caption{Some invariants of the Hopf bundles, see \eqref{eq:BFEinstein} and \eqref{eq:Adiag}.}\label{tab:hopfbundles-numbers}
\begin{tabular}{|c|c|c|c|c|c|c|}
\hline
& $M$\rule[-1.2ex]{0pt}{0pt} \rule{0pt}{2.5ex} & $(n,l)$  & $\zeta$ & $\eta$ & $\Lambda_F$ & $\Lambda_B$ \\
\hline \noalign{\smallskip} \hline 
(i) & $\Ss^{2q+1}$\rule[-1.2ex]{0pt}{0pt} \rule{0pt}{2.5ex} & $(2q+1,1)$ & $1$ & $2q$ & $0$ & $2q+2$ \\
(ii) & $\Ss^{4q+3}$\rule[-1.2ex]{0pt}{0pt} \rule{0pt}{2.5ex} & $(4q+3,3)$  & $3$ & $4q$ & $2$ & $4q+8$ \\
(iii) & $\C P^{2q+1}$\rule[-1.2ex]{0pt}{0pt} \rule{0pt}{2.5ex} & $(4q+2,2)$ & $2$ & $4q$ & $4$ & $4q+8$ \\
(iv) & $\Ss^{15}$\rule[-1.2ex]{0pt}{0pt} \rule{0pt}{2.5ex} & $(15,7)$ & $7$ & $8$ & $6$ & $28$ \\
\hline
\end{tabular}
\end{table}
\end{center}

Direct computations show that the dimensional requirements on families (i), (ii), and (iii) can be further relaxed; that is, bifurcation occurs as $t\nearrow+\infty$ even in some cases where \eqref{eq:etazetaineq} is violated.\footnote{Given the nature of the estimates where \eqref{eq:etazetaineq} is used, it should not be surprising that this condition is only sufficient and not necessary.} Namely, the \emph{conclusion} of Proposition~\ref{prop:bifurcationInfty} still holds on (i) if $6\leq q\leq9$, on (ii) if $q = 2$, and on (iii) if $q = 3$, as can be verified using the computations in the Appendix~\ref{app:berger} to evaluate $\alpha_t$ and $\beta_t$, proceeding as in the proof of Proposition~\ref{prop:bifurcationInfty}. Note that, in these cases, bifurcating branches issuing from $\g_t$, with $t$ sufficiently large, provide examples of nonuniqueness of constant $Q$-curvature metrics in conformal classes $[\g_t]$ with $Y(M,\g_t)<0$ and $Y_4^+(M,\g_t)<0$.

\section{Nonuniqueness on noncompact manifolds via coverings}
\label{sec:noncompact}

The problem of finding metrics with constant $Q$-curvature in a given conformal class can also be studied on noncompact Riemannian manifolds. In this situation, the geometrically natural boundary condition to impose is \emph{completeness} of the metric $\g=u^\frac{4}{n-4}\,\g_0$, which translates to an appropriate growth rate for the conformal factor $u$. Although the constant $Q$-curvature equation remains \eqref{eq:constQcurvequation}, variational formulations such as \eqref{eq:deftotQcurvfunct} are no longer available.
To circumvent this issue and establish nonuniqueness results also in this context, we combine infinite towers of coverings and the reverse Aubin inequality \eqref{eq:AubinTheta4}, in an argument inspired by nonuniqueness results for the Yamabe problem on noncompact manifolds~\cite{frankenstein}.

\begin{proposition}\label{prop:coverings}
Let $(M_0,\g_0)$ be a closed Riemannian manifold of dimension $n\ge5$, and let $(M_\infty,\g_\infty)\to (M_0,\g_0)$ be a Riemannian covering whose group of deck transformations has infinite profinite completion. Suppose $Q_{\g_0}$ is almost positive and $Y(M_0,\g_0)\geq 0$. Then, the conformal class $[\g_\infty]$ contains infinitely many pairwise nonhomothetic complete metrics with constant positive $Q$-curvature.
\end{proposition}

\begin{proof}
It follows from the assumption on the group of deck transformations that there exists an infinite tower of Riemannian coverings
\begin{equation*}
(M_\infty,\g_\infty)\longrightarrow \cdots\longrightarrow (M_k,\g_k)\longrightarrow\cdots\longrightarrow (M_2,\g_2)\longrightarrow (M_1,\g_1)\longrightarrow (M_0,\g_0),
\end{equation*}
where $M_k\to M_0$ is an $\ell_k$-sheeted covering, with $\ell_k\geq 2$ and $\ell_k\nearrow\infty$ as $k\to\infty$, see \cite[Lemma~3.6]{frankenstein}.
Since $(M_k,\g_k)$ are locally isometric to $(M_0,\g_0)$, also $Q_{\g_k}$ is almost positive. Using an observation of Aubin, see Akutagawa and Neves~\cite[Lemma~3.6]{akutagawa-neves}, we have that $Y(M_k,\g_k)>0$ for all $k\geq1$.
Thus, by Proposition~\ref{prop:hangyang},
\[\overline\Theta(M_k,\g)=\frac{\int_{M_k}Q_{\g}\,\vol_\g}{\Vert Q_\g\Vert^2_{L^{\frac{2n}{n+4}}(M_k,\vol_\g)}}\]
attains its maximum in each conformal class $[\g_k]$ at some metric $\widetilde \g_k\in[\g_k]$ for all $k\geq1$.
Denote by $\h_k$ the pullback to $M_k$ of the metric $\widetilde \g_1$, and note that $\h_k\in[\g_k]$.
Since $q_1=Q_{\widetilde\g_1}$ is (a positive) constant and $\h_k$ is locally isometric to $\widetilde\g_1$, then $Q_{\h_k}=q_1$ for all $k\geq1$, and we have:
\begin{align*}
\overline\Theta(M_k,\h_k)&=\frac{\int_{M_k}Q_{\h_k}\,\vol_{\h_k}}{\left(\int_{M_k}Q_{\h_k}^{\frac{2n}{n+4}}\,\vol_{\h_k}\right)^\frac{n+4}n}\\
&=q_1^{-1}\Vol(M_1,\widetilde\g_1)^{-\frac{4}n}\,\left(\frac{\ell_k}{\ell_1}\right)^{-\frac{4}{n}}\searrow0,\ \text{as $k\to\infty$}.
\end{align*}
Thus, $\frac2{n-4}\overline\Theta(M_k,\h_k)<\Theta_4(\Ss^n,\gr)$ if $k\geq k_0$ for some $k_0\geq1$ sufficiently large, and hence $\h_{k_0}$ is not a maximizer of $\Theta_4$ in $[\g_{k_0}]$. By Proposition~\ref{prop:hangyang}, there exists a maximizer of $\Theta_4$ in $[\h_{k_0}]$, not homothetic to $\h_{k_0}$. The pullbacks to $M_\infty$ of this maximizer and $\h_{k_0}$ are two nonhomothetic complete metrics with constant positive $Q$-curvature in $[\g_\infty]$. Infinitely many such metrics can be produced by iterating this procedure, replacing $(M_1,\g_1)$ with $(M_{k_0},\g_{k_0})$.
\end{proof}


We are now ready to prove Theorem~\ref{mainthm:B} in the Introduction, using Proposition~\ref{prop:coverings}.

\begin{proof}[Proof of Theorem~\ref{mainthm:B}]
By classical result of Borel~\cite{borel}, every symmetric space $N$ of noncompact type admits irreducible compact quotients $\Sigma=N/\Gamma$, and the same is true for the Euclidean space $\R^d$. Let us fix such a compact quotient $(\Sigma,\h_\Sigma)$ with the induced locally symmetric metric. Since $\pi_1(\Sigma)$ is infinite and residually finite, it has infinite profinite completion.
The conclusion follows from Proposition~\ref{prop:coverings} with $(M_\infty,\g_\infty)=(C\times N,\g\oplus\h)$,
and $(M_0,\g_0)=(C\times\Sigma,\g\oplus\h_\Sigma)$, which clearly has nonnegative Yamabe invariant and almost positive $Q$-curvature.
\end{proof}




\appendix
\section{\texorpdfstring{The $Q$-curvature of Berger metrics}{The Q-curvature of Berger metrics}}\label{app:berger}

For the convenience of the reader, we now provide explicit formulae for the $Q$-curvature of the canonical variation $\g_t=t\,\g_{\mathcal V}+\g_{\mathcal H}$, $t>0$, for each of the Hopf bundles listed in Table~\ref{tab:hopfbundles}. These metrics are commonly referred to as \emph{Berger metrics}. Further details about such metrics can be found in \cite[\S 9.81-\S 9.85]{besse} and \cite{bp-calcvar}.

We begin by recalling that, since $(\Ss^n,\gr)$ has Ricci tensor $\Ric=(n-1)\id$, its scalar curvature is $\scal=n(n-1)$ and its $Q$-curvature is $Q=\tfrac{1}{8}n(n^2-4)$.

\subsection{\texorpdfstring{Berger spheres $(\Ss^{2q+1},\g_t)$}{Berger spheres with 1-dimensional fiber}}
Consider the Hopf bundle (i) in Table~\ref{tab:hopfbundles} and the canonical variation $\g_t$, where $(\Ss^{2q+1},\g_1)$ is the unit round sphere. The Ricci tensor of $(\Ss^{2q+1},\g_t)$ has eigenvalues
\begin{equation*}
\begin{aligned}
&2qt& &\text{ with multiplicity } 1,\\
&2q+2-2t& &\text{ with multiplicity } 2q,
\end{aligned}
\end{equation*}
hence the Riemannian submersion $\pi_t\colon (\Ss^{2q+1},\g_t)\to\C P^q$ is horizontally Einstein, with $\kappa_t=2q+2-2t$, and furthermore one can explicitly compute:
\begin{align*}
\big\|\Ric_{\g_t}\big\|^2&=(2qt)^2+2q(2q+2-2t)^2,\\
\scal_{\g_t}&=2q(2q+2-t),
\end{align*}
\begin{align*}
Q_{\g_t}&=-\tfrac{2}{(2q-1)^2}\big\|\Ric_{\g_t}\big\|^2+\tfrac{(2q+1)^3-4(2q+1)^2+16(2q+1)-16}{8(2q)^2(2q-1)^2}\scal_{\g_t}^2\\
&=\tfrac{8q^3-68q^2-106q-3}{8(2q-1)^2}\, t^2-\tfrac{8q^4+4q^3-46q^2-45q-3}{2(2q-1)^2}\, t+\tfrac{(2q^2+3q+1)^2(2q-3)}{2(2q-1)^2}.
\end{align*}
A simple analysis of the above coefficients shows the following asymptotic behavior:
\begin{equation*}
\lim_{t\searrow0} Q_{\g_t}<+\infty, \text{ for all } q\geq1, \quad \text{ and }\quad \lim_{t\nearrow+\infty} Q_{\g_t}=\begin{cases}
-\infty, & \text{ for all } q\leq 9\\
+\infty, & \text{ for all } q\geq 10\\
\end{cases}.
\end{equation*}

\subsection{\texorpdfstring{Berger spheres $(\Ss^{4q+3},\g_t)$}{Berger spheres with 3-dimensional fiber}}
Consider the Hopf bundle (ii) in Table~\ref{tab:hopfbundles} and the canonical variation $\g_t$, where $(\Ss^{4q+3},\g_1)$ is the unit round sphere. The Ricci tensor of $(\Ss^{4q+3},\g_t)$ has eigenvalues
\begin{equation*}
\begin{aligned}
&\tfrac2t+4qt& &\text{ with multiplicity } 3,\\
&4q+8-6t& &\text{ with multiplicity } 4q,
\end{aligned}
\end{equation*}
and hence the Riemannian submersion $\pi_t\colon (\Ss^{4q+3},\g_t)\to\Hr P^q$ is horizontally Einstein Riemannian submersion with $\kappa_t=4q+8-6t$, and one can explicitly compute:
\begin{equation*}
\begin{aligned}
\big\|\Ric_{\g_t}\big\|^2&=12(\tfrac1t+2qt)^2+16q(2q+4-3t)^2,\\
\scal_{\g_t}&=2\left(\tfrac3t+8q(q+2)-6qt\right),\\
\end{aligned}
\end{equation*}
\begin{equation*}
\begin{aligned}
Q_{\g_t}&=-\tfrac{2}{(4q+1)^2}\big\|\Ric_{\g_t}\big\|^2+\tfrac{(4q+3)^3-4(4q+3)^2+16(4q+3)-16}{8(4q+2)^2(4q+1)^2}\scal_{\g_t}^2\\
&=\tfrac{3(4q-1)^2(12q+5)}{8(4q+1)^2(2q+1)^2}\,\tfrac{1}{t^2}+\tfrac{(64q^3+80q^2+76q+23)(6q^2+12q)}{(4q+1)^2(2q+1)^2}\,\tfrac{1}{t}\\
&\quad+\tfrac{1024q^7+5376q^6+9408q^5+4656q^4-3600q^3-5100q^2-1423q}{2(4q+1)^2(2q+1)^2}\\
&\quad-\tfrac{(64q^4+80q^3-52q^2-105q-32)(12q^2+24q)}{(4q+1)^2(2q+1)^2}\, t +\tfrac{(48q^3-40q^2-169q-64)(12q^2+9q)}{2(4q+1)^2(2q+1)^2}\, t^2.
\end{aligned}
\end{equation*}
Again, a simple analysis of coefficients shows that:
\begin{equation*}
\lim_{t\searrow0} Q_{\g_t}=+\infty, \text{ for all } q\geq1, \quad \text{ and }\quad \lim_{t\nearrow+\infty} Q_{\g_t}=\begin{cases}
-\infty, & \text{ for all } q\leq 2\\
+\infty, & \text{ for all } q\geq 3\\
\end{cases}.
\end{equation*}

\subsection{\texorpdfstring{Berger metrics $(\C P^{2q+1},\g_t)$}{Berger metrics on complex projective spaces}}
Consider the Hopf bundle (iii) in Table~\ref{tab:hopfbundles} and the canonical variation $\g_t$, where $(\C P^{2q+1},\g_1)$ is the Fubini-Study metric. The Ricci tensor of $(\C P^{2q+1},\g_t)$ has eigenvalues
\begin{equation*}
\begin{aligned}
&\tfrac{4}{t}+4qt& &\text{ with multiplicity } 2,\\
&4q+8-4t& &\text{ with multiplicity } 4q
\end{aligned}
\end{equation*}
and hence the Riemannian submersion $\pi_t\colon (\C P^{2q+1},\g_t)\to\Hr P^q$ is horizontally Einstein Riemannian submersion with $\kappa_t=4q+8-4t$, and one can explicitly compute:
\begin{equation*}
\begin{aligned}
\big\|\Ric_{\g_t}\big\|^2&=\tfrac{32}{t^2}+64 q \left(q^2+4q+5\right)-128 (q^2+2q) t +32 q (q+2) t^2,\\
\scal_{\g_t}&=2(\tfrac{4}{t}+4qt)+4q(4q+8-4t),
\end{aligned}
\end{equation*}
\begin{align*}
Q_{\g_t}&=\tfrac{8 \left(4 q^2-6 q-1\right)}{q (4 q+1)^2 }\tfrac{1}{t^2}+\tfrac{16 \left(8 q^4+20 q^3+14 q^2+13 q+2\right)}{q (4 q+1)^2 }\tfrac{1}{t}\\
&\quad+\tfrac{8 \left(16 q^6+72 q^5+92 q^4+2 q^3-61 q^2-42 q-6\right)}{q (4 q+1)^2}\\
&\quad+16\left(1-\tfrac{8 q^4+20 q^3+14 q^2+13 q+2}{(4 q+1)^2}+\tfrac{2}{q}\right) t-4\left(1-\tfrac{8 q^3+4 q^2+6 q+1}{(4 q+1)^2}+\tfrac{2}{q}\right) t^2.
\end{align*}
As before, a simple analysis of coefficients shows that:
\begin{equation*}
\lim_{t\searrow0} Q_{\g_t}=\begin{cases}
-\infty, & \text{ if } q=1\\
+\infty, & \text{ for all } q\geq 2\\
\end{cases}, \quad \text{ and }\quad \lim_{t\nearrow+\infty} Q_{\g_t}=\begin{cases}
-\infty, & \text{ for all } q\leq 3\\
+\infty, & \text{ for all } q\geq 4\\
\end{cases}.
\end{equation*}

\subsection{\texorpdfstring{Berger spheres $(\Ss^{15},\g_t)$}{Berger spheres with 7-dimensional fiber}}
Consider the Hopf bundle (iv) in Table~\ref{tab:hopfbundles} and the canonical variation $\g_t$, where $(\Ss^{15},\g_1)$ is the unit round sphere. The Ricci tensor of $(\Ss^{15},\g_t)$ has eigenvalues
\begin{equation*}
\begin{aligned}
&\tfrac{6}{t}+8t& &\text{ with multiplicity } 7,\\
&28-14t& &\text{ with multiplicity } 8,
\end{aligned}
\end{equation*}
and hence the Riemannian submersion $\pi_t\colon (\Ss^{15},\g_t)\to\Ss^8(1/2)$ is horizontally Einstein Riemannian submersion with $\kappa_t=28-14t$, and one can explicitly compute:
\begin{equation*}
\begin{aligned}
\big\|\Ric_{\g_t}\big\|^2&=\tfrac{252}{t^2} + 6944 - 6272 t + 2016 t^2,\\
\scal_{\g_t}&=\tfrac{42}{t}-56t+224,\\
Q_{\g_t}&=\tfrac{20259}{1352}\tfrac{1}{t^2}+\tfrac{32388}{169}\tfrac{1}{t}+\tfrac{64383}{169}-\tfrac{30640}{169}t+\tfrac{1366}{169}t^2.
\end{aligned}
\end{equation*}
Note that
\begin{equation*}
\lim_{t\searrow0} Q_{\g_t}=+\infty, \quad \text{ and }\quad \lim_{t\nearrow+\infty} Q_{\g_t}=+\infty.
\end{equation*}

\end{document}